\documentclass[12pt]{amsart}

\usepackage {graphicx,enumerate,booktabs}
\usepackage[all]{xy}

\newtheorem{definition}{Definition}
\newtheorem{lemma}{Lemma}[section]
\newtheorem{theorem}[lemma]{Theorem}

\newtheorem{claim*}{Claim}

\numberwithin{equation}{section}

\newcommand{\Cox}{\operatorname{Cox}}
\newcommand{\Pic}{\operatorname{Pic}}
\newcommand{\Tor}{\operatorname{Tor}}

\newcommand{\Osh}{{\mathcal O}}

\newcommand{\pp}{\mathbb{P}}
\newcommand{\zz}{\mathbb{Z}}

\title{\bf Picard-graded Betti numbers and the defining ideals of Cox rings}
\author{Antonio Laface, Mauricio Velasco}
\begin{document}



\maketitle

\section{Introduction}

When studying the geometry of a projective variety $X$ it is often useful to consider its many projective embeddings. In exchange for this flexibility we are left without a natural choice for what ``The homogeneous coordinate ring" of a projective variety should be. Hu and Keel proposed in~\cite{HK}  the following candidate (inspired by work of Cox on toric varieties~\cite{COX}),
\begin{definition} Let $X$ be a smooth projective variety with torsion-free Picard group and let $D_1,\ldots, D_r$ be effective divisors whose classes are a basis of $\Pic(X)$. The {\em Cox ring} of $X$ with respect to this basis is:
\[
\Cox(X) := \bigoplus_{(m_1,\ldots,m_r)\in\zz^r}H^0(X,m_1D_1+\dots+m_rD_r).
\]
with multiplication induced by the multiplication of functions in $k(X)$.
\end{definition}
Different choices of basis yield (non-canonically) isomorphic Cox rings and any of them is a ``total" coordinate ring for $X$ in the sense that,
\begin{enumerate}
\item{The homogeneous coordinate rings $\bigoplus_{s=0}^\infty H^0(X,sD)$ of all images of $X$ via complete linear systems $\phi_D:X\rightarrow \mathbb{P}(H^0(X,D))$ are subalgebras of $\Cox(X)$.}
\item{If the Cox ring is a finitely generated $k$-algebra then $X$ can be obtained as a quotient of an open set of $\rm{ Spec}(\rm{Cox}(X))$ by the action of a torus. In this sense the points of  $\rm{ Spec}(\rm{Cox}(X))$ can be thought of as homogeneous coordinates for the points in $X$.}
\end{enumerate}

If $\Cox(X)$ is a finitely generated $k$-algebra, the birational geometry of $X$ is especially well structured: the nef and effective cones are polyhedral and there are finitely many small modifications of $X$ satisfying certain mild restrictions (see Proposition 1.11 in ~\cite{HK} for precise statements). The  varieties $X$ with finitely generated Cox rings are called Mori Dream spaces and have been the focus of much interest (see~\cite{BP},\cite{CT},\cite{C},\cite{BCHM}).

When $X$ is a Mori Dream Space, its Cox ring admits a presentation $\Cox(X)\cong k[x_1,\dots, x_n]/I$. Not much is known about the generators of the ideals $I$ in general although some work has been done on specific classes of varieties (see~\cite{BP},\cite{STV},\cite{SS}). The purpose of this paper is to introduce a tool to study the $\Pic(X)$-degrees of the generators of the ideal $I$. 

More precisely, we introduce complexes of vector spaces whose homology determines the structure of the minimal free resolution of $\Cox(X)$ over the polynomial ring and show how the homology of these complexes can be studied by purely geometric methods. As an application of these techniques we give a simple new proof of a characterization of the Cox rings of Del Pezzo surfaces (of degree $>1$) conjectured by Batyrev and Popov in~\cite{BP} and shown by Serganova and Skorobogatov in~\cite{SS}.

\section{A geometry of syzygies}
\label{sec:method}
Let $X$ be a smooth projective variety with $\Pic(X)\cong\mathbb{Z}^r$ and fix  a collection of effective divisors $D_1,\dots, D_r$ whose classes form a basis of $\Pic(X)$. With this choice of basis we will think of divisors on $X$ as expressions $m_1D_1+\dots+m_rD_r$ with $m_i\in \mathbb{Z}$ (note that every linear equivalence class contains exactly one divisor of this form).

If $\Cox(X)$ is a finitely generated $k$-algebra, we will fix a presentation
\[
k[x_1,\dots, x_n]/I\cong\Cox(X)
\]
where
\begin{enumerate}
\item{The images of the variables $x_i$ are irreducible and homogeneous elements of $\Cox(X)$ (we will always denote a variable and it's image with the same symbol).}
\item{The ring $R=k[x_1,\dots, x_n]$ is a $\Pic(X)$-graded polynomial ring.  The grading is obtained by assigning to each variable $x_i$ the Picard degree $C_i\in\mathbb{Z}^r$ of it's image in $\Cox(X)$.}
\item{The ideal $I$ is $\Pic(X)$-homogeneous and prime.}
\end{enumerate}
\begin{lemma}\label{lem:positivity} The $\Pic(X)$-graded polynomial ring $R=k[x_1,\dots , x_n]$ is positively graded.
\end{lemma}
\begin{proof} Assume $\dim(X)=n$ and let $H$ be an ample divisor on $X$. The intersection number ${\rm Deg}(x_i)=H^{n-1}\cdot [x_i]>0$ extends to a positive $\mathbb{Z}$-grading on $R$ which is coarser than the $\Pic(X)$-grading. 
\end{proof}
Since $R$ is positively graded, every finitely generated $\Pic(X)$-graded $R$-module has a unique minimal $\Pic(X)$-graded free resolution. For the module $\Cox(X)$ this resolution is of the form
\[
\mathbb{G}:\dots\rightarrow \bigoplus_{D\in Pic(X_r)}R(-D)^{b_{2,D}}\rightarrow \bigoplus_{D\in Pic(X_r)}R(-D)^{b_{1,D}}\rightarrow R\rightarrow 0 
\]
where the rightmost nonzero map is given by a row matrix whose entries are a set of minimal generators of the ideal $I$. Since the differential of the resolution has degree $0$ we see that $I$ has exactly $b_{1,D}(\Cox(X))$ minimal generators of Picard degree $D$.

We will study the Betti numbers $b_{i,D}(\Cox(X))$ via the following sequence of vector spaces,
\begin{definition} For a divisor $D=m_1D_1+\dots+m_rD_r$, let $\mathbb{A}(D)$ be the sequence
\[
\mathbb{A}_0=H^0(X,D)
\]
\[
\mathbb{A}_j=\bigoplus _{1\leq i_1<\dots<i_j\leq n} H^0(X,D-C_{i_1}-\dots -C_{i_j})\text{ for $1\leq j\leq n$}
\]
with differential $d_j:\mathbb{A}_j\rightarrow \mathbb{A}_{j-1}$ given by
\[
d_1(u_i)=x_iu_0
\]
\[
d(u_{i_1\dots i_j})=\sum_{s=1}^j (-1)^{s+1}x_{i_s}u_{i_1\dots i_{s-1}\hat{i_s}i_{s+1}\dots i_j}
\]
\end{definition}

\begin{lemma} The sequence $(\mathbb{A}(D),d)$ is a complex and $b_{i,D}(Cox(X))=dim_k(H_i(\mathbb{A}(D)))$.
\end{lemma}
\begin{proof} $\mathbb{A}(D)$ is the degree $D$ part of $\Cox(X)\otimes_R \mathbb{K}$ where $\mathbb{K}$ is the Koszul complex on $x_1,\dots, x_n$. Hence
\[
H_i((Cox(X)\otimes \mathbb{K})_D)=(H_i(Cox(X)\otimes \mathbb{K}))_D=(Tor^R_i(Cox(X),k))_D=k^{b_{i,D}(\Cox(X))}
\]
where the last two equalities follow since $\Tor^R(A,B)$ is symmetric in $A$ and $B$ and the Koszul complex is the minimal free resolution of $k$ over $R$.
\end{proof}
The homology of the complexes $\mathbb{A}(D)$ can, in some cases, be determined by purely geometric methods. In this paper we will focus on the calculation of the first Betti numbers and particularly in finding conditions on a divisor $D$ which ensure $b_{1,D}(\Cox(X))=0$.

From now on we will focus on $H_1(\mathbb{A}(D))$ and use the following terminology:
\[
\bigoplus_{1\leq i<j\leq n}H^0(X,D-C_i-C_j)\rightarrow \bigoplus_{i=1}^n H^0(X,D-C_i)\rightarrow H^0(X,D)
\]
\begin{itemize}
\item{A cycle is an expression $\sigma=\sum_{i=1}^n s_iu_i$ with $s_i\in H^0(X,D-C_i)$ such that $\partial\sigma=\sum s_ix_i=0\in H^0(X,D)$.}
\item{A boundary is a linear combination of expressions of the form $s_{kt}(x_ku_t-x_tu_k)\in \bigoplus_{i=1}^n H^0(X,D-C_i)$ with coefficients $s_{kt}\in H^0(X,D-C_k-C_t)$. Note that modulo boundaries $(s_{kt}x_k)u_t=(s_{kt}x_t)u_k$ so any section in direction $u_t$ divisible by $x_k$ can be substituted by one in direction $u_k$ divisible by $x_t$}.
\item{The support of a cycle $\sigma=\sum_{i=1}^n s_iu_i$ is $||\sigma||=\{i:  s_i\neq 0\}$ and the size of the support is the cardinality of $||\sigma||$ denoted  $|\sigma|$.}
 \end{itemize}
\begin{lemma} \label{lem:easy} Any cycle $\sigma=s_iu_i+s_ju_j$ with $|\sigma|\leq 2$ is a boundary\end{lemma}
\begin{proof} If $\sigma$ is a cycle with $|\sigma|\leq 1$ then $\sigma=0$ since $\Cox(X)$ is an integral domain. If $\sigma$ is a cycle with $|\sigma|=2$ then $0=\partial\sigma=s_ix_i+s_jx_j$ so $s_ix_i=-s_jx_j$. Since  $\Cox(X)$ is factorial (see~\cite{EKW}) and $x_i,x_j$ are irreducible it follows that $s_i=s'x_j$ and $s_j=s'x_i$  so $\sigma=s'(x_iu_j-x_ju_i)=\partial(s'u_{ij})$\end{proof}


\begin{lemma} \label{lem:surj} If $C_1$ and $C_2$ are disjoint and $D$ is such that $H^1(X,D-C_1-C_2-C_3)=0$ then the map $s_1u_1+s_2u_2\rightarrow s_1x_1+s_2x_2$ 
\[
H^0(X,D-C_1-C_3)\oplus H^0(X,D-C_1-C_2)\rightarrow H^0(X,D-C_1)
\]
is surjective.
\end{lemma}
\begin{proof} Since $C_1$ and $C_2$ are disjoint the following short sequence of sheaves on $X$ is exact,
\[
0\rightarrow \Osh_X[-C_1-C_2]\rightarrow \Osh_X[-C_1]\oplus \Osh_X[-C_2]\rightarrow \Osh_X\rightarrow 0
\]
Tensoring with $\Osh_X[D]$ the long exact sequence in sheaf cohomology gives 
\[
\dots\rightarrow 
\bigoplus_{i=1}^2 H^0(X,D-C_i-C_3) \rightarrow H^0(X,D-C_3)\rightarrow H^1(X,D-C_1-C_2-C_3)\rightarrow\cdots
\]
from which the statement follows immediately.
\end{proof}
To show that all cycles $\sigma$ in $\mathbb{A}_1(D)$ for some divisor $D$ are boundaries (and conclude that $b_{1,D}(\Cox(X))=0$) we will perform two steps:
\begin{enumerate}
\item{We will use Lemma~\ref{lem:surj} to modify the support of cycles. More precisely,  if we know that every section $s_3$ in $H^0(X,D-C_3)$ can be written as $p_1x_1+p_2x_2$ then, modulo boundaries, \[
s_3u_3=p_1x_1u_3+p_2x_2u_3=p_1x_3u_1+p_2x_3u_2
\]  
so we can ``remove" the $u_3$ component from any cycle $\sigma$ (by adding components in directions $u_1$ and $u_2$ in exchange). For each divisor $D$, a set of these reductions is allowed and we will characterize them explicitly.}
\item{We will show that, if $D$ admits enough reduction moves, these can be combined to strictly reduce the size of the support of any cycle $\sigma$. Using Lemma~\ref{lem:easy} we will conclude that $\sigma$ is a boundary. To keep track of the reductions we introduce the language of games on graphs (see section~\ref{games}).}
\end{enumerate}
In the rest of this paper we will use this method to compute the Picard degrees of the minimal generators of the ideals which define the Cox rings of Del Pezzo surfaces.

\section{Del Pezzo surfaces}\label{sec:DelPezzo}
In this section we describe the fundamental traits of the geometry of Del Pezzo surfaces, for a more detailed treatment the reader should refer to~\cite{MANIN}. 
\begin{definition} A collection of $r\leq 8$ points in $\mathbb{P}^2$ is said to be in general position if no three lie on a line, no six lie on a conic and any cubic containing eight points is smooth at each of them.\end{definition}
\begin{definition} A Del Pezzo surface $X$ is a surface isomorphic to $\mathbb{P}^1\times \mathbb{P}^1$ or to the blow up of $\mathbb{P}^2$ at $r\leq 8$ general points $p_1,\dots, p_r$. In the second case we denote the surface by $X_r(p_1,\dots, p_r)$ or just by $X_r$ if the points are clear from the context.\end{definition}
There are at least two reasons to consider Del Pezzo surfaces as a class in themselves:
\begin{itemize}
\item{Del Pezzo surfaces can be characterized as those nonsingular projective surfaces with ample anticanonical divisor.}
\item{If $X$ is a Del Pezzo surface and $f:X\rightarrow Y$ is a birational morphism then $Y$ is also a Del Pezzo surface.}
\end{itemize}
There is exactly one Del Pezzo surface $X_r$ for each $r\leq 4$ (since the automorphism group of $\mathbb{P}^2$ acts transitively on $4$-tuples of general points) and for each $8\geq r\geq 5$ there are infinitely many nonisomorphic $X_r(p_1,\dots, p_5)$. If $r\leq 3$, we can assume that the blown up points are torus invariant and conclude that the surface $X_r$ is toric. The Del Pezzo surfaces $X_6(p_1,\dots, p_6)$ are precisely the cubic surfaces in $\mathbb{P}^3$. 


\begin{lemma}\label{lem:DelPezzobasics} For a Del Pezzo surface $X_r$, blown up from $\mathbb{P}^2$ via $\pi: X_r\rightarrow \mathbb{P}^2$ the following statements hold:
\begin{enumerate}
\item{${\rm Pic}(X_r)\cong \mathbb{Z}^{r+1}$ and a basis is given by
\begin{itemize}
\item{The pullback of the class of a hyperplane in $\Pic(\mathbb{P}^2)$, denoted $L=\pi^*([H])$ }
\item{The $r$ exceptional divisors of the blow ups, denoted $E_1,\dots, E_r$}
\end{itemize}
}
\item{In terms of this basis the intersection form on $\Pic(X_r)$ is:
\[
E_i\cdot E_j=-\delta_{ij}\text{, $L^2=1$ and } L\cdot E_i=0\text{ for all $i$}
\]
where $\delta_{ij}$ equals $1$ if $i=j$ and $0$ otherwise.
}
\item{The canonical divisor on $X_r$ is
\[
K=-3L+E_1+E_2+\dots +E_r
\]}
\end{enumerate}
\end{lemma}

\begin{definition} An irreducible curve $E$ on $X$ is an exceptional curve if and only if
\[
E^2=-1\text{ and } -K\cdot E=1
\]
\end{definition}
The following description of exceptional curves on a Del Pezzo surface can be found in~\cite{MANIN}.
\begin{lemma} The exceptional curves on any Del Pezzo surface $\pi:X_r\rightarrow \mathbb{P}^2$ ($r\leq 7$) are:
\begin{itemize}
\item{The exceptional divisors $e_i$}
\item{The strict transforms of lines through pairs of points $f_{ij}$.}
\item{For $r\geq 5$, the strict transforms of conics through five points. We will denote with $g_{S}$ the quadric going through the points $p_i$ with $i$ in the complement of $S\subset \{1,\dots, r\}$. Thus for $r=6$ these quadrics will be denoted $g_1,\dots, g_6$ and for $r=7$ they will be denoted with $g_{12}, \dots, g_{67}$.}
\item{For $r\geq 7$ the strict transforms $h_i$ of the cubics through $r$ points vanishing doubly through $p_i$.}
\end{itemize}
Their classes in the Picard group are shown in the table below (up to permutation of the $E_i$) 
\begin{table}[h]
\caption{Exceptional curves on a Del Pezzo surface}
\label{tab:-1curves1}
\begin{center}
\begin{tabular}{c|c|l}
$r$ & Symbol & Picard Degree (up to permutation of the $E_i$)\\[2pt]
\hline
\hline
$2,3,4$ & $e_1$ & $E_1$\\[2pt]
& $f_{12}$ & $L-E_1-E_2$ \\[2pt]
$5$, $6$ & $g_{6}/g_{67}$ & $2L-E_1-E_2-E_3-E_4-E_5$\\[2pt]
$7$ & $h_i$ & $3L-2E_1-E_2-E_3-E_4-E_5-E_6-E_7$\\

\end{tabular}
\end{center}
\end{table}

\end{lemma}

\begin{definition} For a divisor $D$ on $X_r$ let 
\[
m_D=\min\{D\cdot F:\text{ $F$ is an exceptional curve on $X_r$}\}
\]
\end{definition}
A divisor on a Del Pezzo surface is nef if and only if $m_D\geq 0$, moreover every nef divisor is effective,
\begin{lemma}\label{lem:nefandeffective} If $r\geq 2$ and $D\in Pic(X_r)$ is such that $D\cdot E_i\geq 0$ for every exceptional curve then $D$ is effective.
\end{lemma}
\begin{proof} If $r=2$, the picard classes of exceptional curves are $E_1, L-E_1-E_2, E_2$. These are a basis for the Picard group ${\rm Pic}(X_2)$ with dual basis (with respect to the intersection form) $L-E_1,L,L-E_2$ consisting of effective divisors. Thus a nef divisor $D$ is a combination with nonnegative coefficients $D=a_0(L-E_1)+a_1L+a_2(L-E_2)$ and hence effective.
If $D$ is a nef divisor on $X_r$ for $r>2$ and $m_D=0$ then $D$ is a divisor on a smaller Del Pezzo and thus effective by induction. If $m_D>0$ let $D=D'-m_DK$ and note that since $-K$ is effective $D$ is effective if $D'$ is. Finally $D'$ is nef and $m_{D'}=0$ so by induction $D'$ is effective.  
\end{proof}

\section{The Cox rings of Del Pezzo surfaces}

The Cox ring of a Del Pezzo surface $X_r$ is the $\Pic(X_r)$-graded algebra

\[
\Cox(X_r)=\hspace{-5pt} \bigoplus_{(m_0,\dots,m_r)\in\mathbb{Z}^{r+1}} \hspace{-5pt} 
H^0(X_r,m_0L+m_1E_1+\dots+m_rE_r)
\]

Batyrev and Popov show in~\cite{BP} that $\Cox(X_r)$ is a finitely generated $k$-algebra for $r\leq 8$ and that it's generators are in one to one correspondence with the exceptional curves on $X_r$ if $3\leq r\leq 7$ (see Theorem 3.2 in~\cite{BP}).

\begin{definition} A global section $\gamma\in H^0(X_r,D)$ is distinguished if its zero locus is supported on exceptional curves.\end{definition}

If $C_j=m_0L+m_1E_1+\dots+m_rE_r$ is the divisor class of an exceptional curve then $H^0(X_r,C_j)$ has only one distinguished global section (up to multiplication by a nonzero constant). Batyrev and Popov show that such sections generate the algebra $\Cox(X_r)$ (see Proposition 3.4 in~\cite{BP}). 

As a result, there is a presentation
\[
k[V_r]/I_r(p_1,\dots, p_r)\cong{\rm Cox}(X_r(p_1,\dots, p_r))
\]
where $k[V_r]$ is a $\Pic(X_r)$-graded polynomial ring with one variable for each exceptional curve in $X_r$. The $\Pic(X_r)$-homogeneous ideal $I_r(p)$ consists of the linear relations among all distinguished global sections of divisors of the form $D=m_0L+m_1E_1+\dots+m_rE_r$ on $X$. By letting ${\rm Deg}(D)=-K\cdot D$ we see that the $\Pic(X)$-grading is finer than the grading by total degree on $k[V_r]$ and in particular that $I_r(p)$ is homogeneous with respect to this grading. Thus the following ideal is well defined,
\begin{definition} Let $Q_r(p)$ be the ideal generated by the total degree $2$ part of $I_r(p)$.
\end{definition}
What are the minimal generators of the ideals $I_r(p_1,\dots, p_r)$?  Batyrev and Popov provide the following conjectural description:
\begin{itemize}
\item{{\bf [BP] Conjecture of Batyrev and Popov:} For $4\leq r\leq 8$ and every choice of points $p_1,\dots, p_r$, the ideal $I_r(p_1,\dots, p_r)$ is generated by quadrics. In other words $I_r(p)=Q_r(p)$.} 
\end{itemize}
The nef and effective divisor classes of anticanonical degree $2$ can be written as a sum $D=F_1+F_2$ of (the picard classes of ) two intersecting exceptional curves (the curves must intersect since otherwise $D\cdot F_1=-1<0$ contradicting the fact that $D$ is nef). Considering all such pairs in $X_r$ we immediately obtain Table~\ref{tab:Q2}. We will show that all minimal generators of $I_r(p)$ have nef and effective degrees so that ${\rm [BP]}$ is equivalent (for $r\leq 7$) to 
\begin{itemize}
\item{{\bf [BP] in terms of $\Pic(X_r)$-graded Betti numbers:} For $r\leq7$ and any Del Pezzo surface $X_r(p_1,\dots, p_r)$, the Betti number $b_{1,D}(\rm{Cox}(X_r))=0$ for all $D\in Pic(X_r)$ with $-K\cdot D\geq 3$.
\smallskip
}
\end{itemize}

As stated in the introduction, the conjecture of Batyrev and Popov was shown for $r\leq 7$ by Serganova and Skorobogatov in~\cite{SS}. In the rest of this paper we provide an elementary proof of this result using the method introduced in Section~\ref{sec:method}. We will begin by showing that all minimal generators of $I_r(p)$ have nef and effective Picard degrees,

\begin{table}
\caption{Nef divisors of anticanonical degree $2$ (up to permutation of the $E_i$)}
\label{tab:Q2}
\begin{center}
\begin{tabular}{c|l}
$r$ &  Picard Degree \\[2pt]
\hline 
\hline
$4,5$ & $L-E_1$\\[2pt]
& $2L -E_1-E_2-E_3-E_4$\\[2pt]
\hline
$6$ & $3L-2E_1-E_2-E_3-E_4-E_5-E_6$\\[2pt]
\hline
$7$ & $4L-2E_1-2E_2-2E_3-E_4-E_5-E_6-E_7$\\[2pt]
& $5L-E_1-2E_2-2E_3-2E_4-2E_5-2E_6-2E_7$\\[2pt]
& $3L-E_1-E_2-E_3-E_4-E_5-E_6-E_7=-K_7$\\ 
\end{tabular}
\end{center}
\end{table}

\begin{lemma} \label{lem:mingens} If $b_{1,D}(Cox(X_r))\neq0$ then $D$ is a nef and effective divisor.\end{lemma}
\begin{proof} If $D$ is not nef then there exists an exceptional curve $E$ such that $D\cdot E\leq -1$. For any exceptional curve $F$ distinct from $E$ we have that $(D-F)\cdot E < 0$ and this implies that:
\[
H^0(X_r,D-E-F)\otimes H^0(X_r,E)\cong H^0(X_r,D-F),
\]
thus every section in $H^0(X_r,D-F)$ is a multiple of $E$.

As a result, modulo boundaries, every cycle $\sigma=\sum s_iu_i$ is equal to a cycle $p_ku_k$ which is congruent to $0$ by Lemma~\ref{lem:easy}. Hence $\sigma$ is a boundary and $b_{1,D}(Cox(X_r))=0$.\end{proof} 

{\bf Example.} For $r=4$, we can assume, without loss of generality that $p_1=[1:0:0]$, $p_2=[0:1:0]$, $p_3=[0:0:1]$ and $p_4=[1:1:1]$. Thus,
\begin{itemize}
\item{$R=k[V_4]=k[e_1,\dots ,e_4,f_{12},\dots,f_{34}]$ graded by
$\deg(f_{ij})=L-E_i-E_j \text{ ,  } \deg(e_i)=E_i$}
\item{$Q_4$ is the ideal generated by
\[
\begin{array}{c|c}
$Degree in $\Pic(X_4)$ $ & $Minimal Generator$\\
\hline
L-E_1 & e_2f_{12} -e_3f_{13}-e_4f_{14},\\
L-E_2 & e_1f_{12} - e_3f_{23}-e_4f_{24},\\
L-E_3 & e_1f_{13} -e_2f_{23}+e_4f_{34},\\
L-E_4 & e_1f_{14} -e_2f_{24}-e_3f_{34}\\
2L-E_1-E_2-E_3-E_4 & f_{14}f_{23}-f_{12}f_{34} -f_{13}f_{24},\\
\end{array}
\]}
\end{itemize}
As observed by Batyrev and Popov in~\cite{BP} the ideal $Q_4$ is the coordinate ring of $Gr(2,5)$ in the Plucker embedding and in particular a prime ideal. The equality $Q_4=I_4$ follows by a dimension argument (see Proposition 4.1 in~\cite{BP}).

\section{The strategy}
We will prove [BP] by induction on $r$. For each $r>4$ we will show:
\begin{enumerate}
\item{$b_{1,D}({\rm Cox}(X_r))=0$ for all nef and effective divisors of degree $> 2$ which have the property that $D\cdot C_j=0$ for some exceptional curve $C_j$. These divisors are pullbacks of divisors on a Del Pezzo $X_{r-1}$ and the statement will follow by the induction hypothesis.}
\item{If a divisor does not contract any exceptional curve, then $H_1(\mathbb{A}(D))=0$. We will show this by writing every cycle in $\mathbb{A}_1(D)$ as a sum of boundaries via an algorithm which uses Lemma~\ref{lem:surj} and the combinatorics of exceptional curves. Developing this algorithm will be the content of the remaining sections.}
\end{enumerate}

We begin with a preliminary lemma about distinguished sections of $X_7$.

\begin{lemma}\label{27-curves}
Let $E_1E_1',\ldots,E_{28},E_{28}'$ be the distinguished sections of $|-K|$ on $X_7$. Any $27$ of these sections span a three dimensional vector space.
\end{lemma}
\begin{proof}
Let $\phi: X_7 \rightarrow \pp^2$ be the double covering induced by $|-K|$. The image $\phi(E)$ of a $(-1)$-curve is a bitangent line to the branch divisor $B$ ( which is a smooth plane quartic ). The dual curve $B^*$ has degree 12 with 28 nodes ( the bitangents of $B$ ) and 24 cusps ( the flexes of $B$ ).
Now 
\[
	\phi^{-1}(\phi(E)) = E + E'  = -K
\]
so that if the sections $A_1A_1',\ldots,A_{27},A_{27}'$ span a 2 dimensional vector space, then all the $\phi(A_i)$'s are bitangent lines which live in a pencil. 
This would imply that $B^*$ has 27 nodes on a line $L^*$ ( the dual of the pencil ) so that $L^*$ splits off from $B^*$ and this is impossible because $B^*$ is irreducible.
\end{proof}

The following lemma carries out step $(1)$. We denote the dimension of a vector space $W$ by $|W|$.
\begin{lemma} \label{lem:DcontractsCurves} If $D\in {\rm Pic}(X_r)$ is nef, contracts some curve and $-K\cdot D\geq 3$ then 
\[
b_{1,D}(Cox(X_r))=0
\]
\end{lemma}

\begin{proof} 
By acting via the Weyl group and changing the presentation of the Del Pezzo surface we can assume, without loss of generality, that $D$ contracts the exceptional curve $e_r$ (i.e. $D\cdot e_r=0$). 
Let $X_{r-1}$ be the Del Pezzo surface obtained by blowing down $e_r$ and $V_{r-1}$ be the set of exceptional curves in $X_r$ which do not intersect $e_r$. Note that $V_{r-1}$ is the set of exceptional curves in $X_{r-1}$. Since $Q_{r-1}\subset Q_r\cap k[V_{r-1}]$ we have a homomorphism of graded $k$-algebras
\[
\psi:k[V_{r-1}]/Q_{r-1}\rightarrow k[V_r]/Q_r
\]

We will show that this map is surjective in multidegree $D$.  
Let $\{y_1,\dots, y_k\}=V_{r-1}$ and let $\{x_1,\dots, x_s\}$ be the exceptional curves in $V_r$ which intersect $e_r$. Let $m=y_1^{a_1}\cdots y_k^{a_k}x_1^{b_1}\cdots x_s^{b_s}e_r^t$ be a monomial in $k[V_r]$ of degree $D$, there are two cases to consider
\begin{itemize}
\item{If $r\leq 6$, then $D\cdot e_r=0$ implies that $t=b_1+\dots+b_s$ so that we can write $m=y_1^{a_1}\cdots y_k^{a_k}(x_1e_1)^{b_1}\cdots (x_se_1)^{b_s}$. Now, $Q_r$ contains relations of the form $x_je_r=\sum c_{ab}y_ay_b$ coming from the conic bundles $deg(x_i)+deg(e_1)$ so the monomials on the $y's$ span $(k[V_r]/Q_r)_D$.}
\item{If $r=7$ either $h_7$ does not appear in the monomial $m$ (and the same reasoning as when $r\leq 6$ shows that $m$ can be written as a linear combination of monomials on the $y_i$) or $h_7$ is one of the $x_i$, say $x_s=h_7$. In this case $D\cdot e_7=0$ implies that $t=b_1+\dots+b_{s-1}+2b_s$  and we can write $m=y_1^{a_1}\cdots y_k^{a_k}(x_1e_7)^{b_1}\cdots (x_{s-1}e_7)^{b_{s-1}}(h_7e_7^2)^{b_s}$.
Now $h_7e_7$ is a distinguished section of $-K$ so, by lemma~\ref {27-curves}, there is a relation of the form $h_7e_7=a_1A_1A_1'+\dots+a_3A_3A_3'$ so
\[
h_7e_7^2=a_1(A_1e_7)A_1'+\dots +a_3(A_3e_7)A_3'
\]
where we can relabel the curves so that the $A_i$ are adjacent to $e_7$ (recall that for every exceptional curve $B$ in $X_7$, $B':=-K-B$ is the picard class of an exceptional curve and exactly one of $B$ and $B'$ are adjacent to $e_7$). Hence $m$ is expressible as a linear combination of monomials which do not contain $h_7$ and the same argument as for $r\leq 6$ shows that the monomials on the $y's$ span $(k[V_r]/Q_r)_D$.}
\end{itemize}
We conclude that $\psi$ is surjective in degree $D$. By construction $k[V_r]/Q_r$ surjects onto $k[V_r]/I_r={\rm Cox}(X_r)$ and by induction on $r$ we know that ${\rm Cox}(X_{r-1})=k[V_{r-1}]/Q_{r-1}$. thus
\[
|{\rm Cox}(X_{r-1})_D|=|k[V_{r-1}]/Q_{r-1}|\geq |(k[V_r]/Q_{r})_D|\geq|k[V_r]/I_r|=|{\rm Cox}(X_r)_D|
\]
but $h^0(D,X_r)=h^0(D,X_{r-1})$ since $D$ contracts $e_r$ so $(k[V_r]/Q_r)_D=(k[V_r]/I_r)_D$ and $I_r$ cannot have minimal generators in multidegree $D$ with $-K\cdot D\geq 3$.    
\end{proof}

\section{Games on the graphs of exceptional curves}\label{games}
We will now introduce a combinatorial tool for studying the first homology of the complexes $\mathbb{A}(D)$. The term graph will mean finite graph without loops (multiple edges are allowed). 
\begin{definition}
For a graph $G$ let $V(G)$ be the set of vertices of $G$. If $S=\{v_1,\dots, v_k\}\subset V(G)$ the graph induced by $S$ is the graph $H$ with vertex set $S$ and with edge set given by all edges of $G$ whose endpoints are in $S$.
\end{definition}
\begin{definition} A capture diagram is a graph $H$ with $V(H)=\{a_1,a_2,c\}$. The distinguished vertex $c$ is called the captured vertex.\end{definition}
\begin{definition} A capture move for the diagram $H$ on the graph $G$ is a morphism of graphs
\[
\phi:H\rightarrow G
\] such that $H$ is isomorphic via $\phi$ to the subgraph induced by $\{\phi(a_1),\phi(a_2),\phi(c)\}$.
If $G'$ is the subgraph of $G$ induced by $V(G)-\phi(c)$ we say that $G$ can be captured from $G'$ with the move $(H,\phi)$ and denote it with $G'\rightarrow G$. 
\end{definition} 
\begin{definition} A graph $G$ is 2-capturable using the diagrams $\{H_1,\dots, H_s\}$ if there is a sequence of induced subgraphs
\[
G_0\rightarrow G_1\rightarrow G_2\rightarrow \dots \rightarrow G_k=G
\]
such that $|V(G_0)|\leq 2$ and for all $i$, $G_i\rightarrow G_{i+1}$ is a capture move for some diagram in $\{H_1,\dots, H_s\}$. We will call $G_0$ the initial or starting subgraph.
\end{definition}
We will be interested in the following graphs
\begin{definition} The graph of exceptional curves in $X_r$, denoted $G_r$ is the graph with vertices  the exceptional curves in $X_r$ and with $C_i\cdot C_j$ edges between vertices $C_i$ and $C_j$ (thus $G_r$ has multiple edges for $r\geq 7$).\end{definition}

We will now study the relationship between capturability and Picard Betti numbers,
\begin{definition} A move $\phi:H\rightarrow G_r$ on the graph of exceptional curves is valid for a divisor $D$ on $X_r$ if the map
\[
H^0(X_r,D-\phi(c)-\phi(a_1))\oplus H^0(X_r,D-\phi(c)-\phi(a_2))\rightarrow H^0(X_r,D-\phi(c))
\]
is surjective. That is, every section $s\in H^0(X_r,D-\phi(c))$ can be written as $s=s_1\phi(a_1)+s_2\phi(a_2)$
\end{definition}

\begin{theorem}\label{thm:games} If $G_r$ is $2$-capturable using moves that are valid for $D$ then 
\[
b_{1,D}({\rm Cox}(X_r))=0
\]
\end{theorem}
\begin{proof}  Let $\sigma=\sum p_iu_i\in \bigoplus_{i=1}^nH^0(X_r,D-C_i)$ be a cycle. And assume
\[
S_0\rightarrow S_1\rightarrow\dots\rightarrow S_k=G_r
\]
is a sequence of capture moves with $|V(S_0)|\leq 2$. Let $\phi: H\rightarrow S_k$ be the last capture move and let $C_3=\phi(c)$, $C_1=\phi(a_1)$ and $C_2=\phi(a_2)$. Note that $C_1,C_2\in V(S_{k-1})$ and that there are $s_1,s_2$ such that $p_3=s_1x_1+s_2x_2$ (since the move $\phi$ is valid for $D$). Thus  
\[
\sigma=(\sum_{i\neq 3}p_iu_i)+p_3u_3=\sum_{i\neq 3}p_iu_i+(s_1x_1+s_2x_2)u_3
\]
Modulo boundaries the above expression equals
\[
\sum_{i\neq 3}p_iu_i+s_1x_3u_{1}+s_2x_3u_{2}
\]
which is a cycle whose support is contained in $V(S_{m-1})$. Continuing inductively we see that modulo boundaries $\sigma=\tau$ where $\tau$ is a cycle whose support is contained in $S_0$ (and hence of size $\leq 2$). By Lemma~\ref{lem:easy}, $\sigma$ is a boundary.
\end{proof}
We will use the Kawamata-Viehweg vanishing theorem to show that, if $D$ is sufficiently positive then there are enough capture moves to guarantee that $b_{1,D}(\Cox(X))=0$ via Theorem~\ref{thm:games}.

\section{Valid moves on the Del Pezzo surfaces $X_5(p)$ and $X_6(p)$}
We will use the following easily verified facts about the graph of exceptional curves on the Del Pezzo surfaces $X_5(p)$ and $X_6(p)$
\begin{enumerate}
\item{If $r=5$ and $A,B$ are two adjacent exceptional curves then there exist adjacent curves $A'$ and $B'$ such that the induced subgraph on $A,B,A',B'$ is a square and $A+B+A'+B'=-K$.  Moreover there are no triangles in $G_5$.}
\item{If $r=6$ and $A,B$ are two adjacent exceptional curves there is a unique curve $C$ such that $A+B+C=-K$. Moreover, the induced subgraph on $A,B,C$ is a triangle.}
\end{enumerate}

\begin{lemma} \label{lem:subgraphs}
Suppose $r=5,6$ and let $\{A,B,C\}$ be exceptional curves on $X_r$. If either
\begin{itemize}
\item{$A\cdot B=1$, $A\cdot C=B\cdot C=0$ and $r= 5$}
\item{$A\cdot B=B\cdot C=1$, $A\cdot C=0$ and $r= 6$}
\end{itemize}
Then every other exceptional curve intersects at most two curves in $\{A,B,C\}$.
\end{lemma}
\begin{proof} For $r\leq 5$ the graph of exceptional curves contains no triangles. For $r=6$ every edge of this graph belongs to exactly one triangle since two distinct lines in $\mathbb{P}^3$ span a hyperplane which intersects a cubic surface in a curve of degree $3$ containing two lines, (i.e. a triangle) thus determining the third line. 
\end{proof}

\begin{lemma} \label{lem:m_D>0} For $r=5,6$ let $D$ be a divisor in $X_r$ with $m_{D}\geq 1$. Then the moves in Lemma~\ref{lem:subgraphs} are valid for $D$.\end{lemma}
\begin{proof} let $L=D-A-B-C-K$. By Lemma~\ref{lem:subgraphs} any exceptional curve $F$ in $X_r$ has $(A+B+C)\cdot F\leq 2$, thus $L\cdot F\geq 1-2+1\geq 0$ and $L$ is nef and effective. We will show that $L$ is also big so that, by Kawamata-Viehweg $H^1(X_r,L)=0$. The validity for $D$ of the moves will then follow from Lemma~\ref{lem:surj}. We show that $L^2>0$ by consider two cases:  
\begin{itemize}
\item{  If $r=5$, $A\cdot B=1$ and $B\cdot C=A\cdot C=0$, let $A'$ and $B'$ be two exceptional curves such that $-K=A+B+A'+B'$ and note that these curves form a square. Thus $L=D-C+A'+B'$ and $L^2=(D-C)^2+2(D-C)(A'+B')$. Now $D-C$ is nef and effective (since $m_D\geq 1$) and $(D-C)(A'+B')\geq 1$ (since no (-1)-curve intersects $A'$ and $B'$ simultaneously). Hence $L^2>0$}
\item{ If $r=6$ and $A\cdot B=B\cdot C=1$ let $C'$ be the unique (-1)-curve such that $-K=A+B+C'$ and note that these curves form a triangle and that $C\cdot C'=0$ (by lemma~\ref{lem:subgraphs}). Thus $L=D-C+C'$ and $L^2=(D-C)^2+2(D-C)C'-1\geq 0+2-1=1$ since $D-C$ is nef and effective.}
\end{itemize}
\end{proof}

\begin{lemma} \label{lem:capture} For $r=5,6$ the graphs $G_r$ are $2$-capturable with starting subgraph $S_0=\{e_1,e_2\}$ using the capture moves of Lemma~\ref{lem:subgraphs}\end{lemma}
\begin{proof} We will list the exceptional curves captured at every stage (see Table~\ref{tab:-1curves1} for the interpretation of the variables)
\\ 
For $r=5$\\
\[
\begin{array}{l|l}
S_0 &  e_1,e_2\\
 & f_{13}, f_{14}, f_{15}, f_{23}, f_{24}, f_{25}\\
 & g, e_3, e_4, e_5, f_{12}, f_{34}, f_{35}, f_{45}\\
\end{array}
\] 
For $r=6$\\ 
\[
\begin{array}{l|l}
S_0 &  e_1,e_2\\
 & f_{12}, g_3, g_4, g_5, g_6\\
 & \text{ All remaining curves}\\
\end{array}
\]
\end{proof}
\begin{lemma} \label{lem:X6mD0} For $r=5,6$ and every divisor $D$ with $m_D\geq 1$ 
\[
b_{1,D}({\rm Cox}(X_r))=0
\]
\end{lemma}
\begin{proof} This follows immediately from Theorem~\ref{thm:games} and Lemmas~\ref{lem:m_D>0} and~\ref{lem:capture}.
\end{proof}
Note that for $r=5,6$ there are infinitely many nonisomorphic Del Pezzo surfaces (depending on the choice of coordinates for the blown up points) and thus infinitely many nonisomorphic Cox rings (since the surface can be recovered from the ring), the last Lemma shows that all these have no minimal generators in degrees $D$ with $m_D\geq 1$ simultaneously.

In the next section we will use the same argument on the Del Pezzo surfaces $X_7$. We have chosen to write it in a different section in order to simplify the exposition.

\section{Valid moves on the Del Pezzo surfaces $X_7(p)$}
We will use the following easily verified facts about exceptional curves on a Del Pezzo surface $X_7$.
\begin{enumerate}
\item{The anticanonical divisor determines a fixed point free involution on the set of exceptional curves via $A\rightarrow A'=-K-A$. We will call $A'$ the dual of $A$. Thus the $h_i$ and the $e_i$'s are duals of each other and so are $f_{ij}$ and the $g_{ij}$. Moreover note that two exceptional curves $A$ and $B$ have the property that $A\cdot B\geq 2$ iff $B=A'$ and in that case $A\cdot B=2$. Thus the graph of exceptional curves $G_7$ is a multigraph in which every vertex belongs to exactly one double edge. }
\item{$A$ is adjacent to $B$ iff $A'$ is disjoint from $B$}
\item{There are no triangles containing a double edge.}
\item{The sum of the curves in any triangle is $-K+V$ for some exceptional curve $V$. The sum of two doubly adjacent curves is $-K$}
\item{${\rm diam}(G_7)=2$}
\end{enumerate}
We will divide the nef and effective divisors which do not contract curves into classes according to the value of $m_D$.
It will be sufficient to study two cases: $m_D\geq 2$ and $m_D=1$. 

\subsection{ Divisors $D$ with $m_D\geq 2$.}

\begin{figure}[h]
\begin{center}
\scalebox {0.75}{
\includegraphics{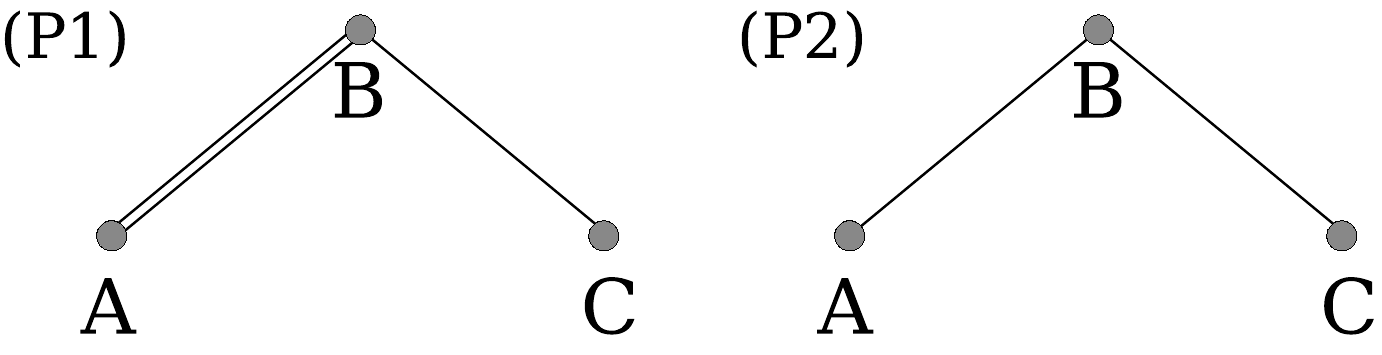}
}
\caption{Valid moves for $r=7$ and $m_D\geq 2$}
\label{fig:moves0}
\end{center}
\end{figure}
\begin{lemma} \label{lem:configs0}
 If $A, B$ and $C$ are exceptional curves in $X_7$ then the following capture moves (see Figure~\ref{fig:moves0}) are valid for any divisor $D$ with $m_D\geq 2$.
\begin{itemize}
\item{(P1): $B=A'$}
\item{(P2): $B\cdot C=B\cdot A=1$}
\end{itemize}
\end{lemma}
\begin{proof} By Kawamata-Viehweg and Lemma~\ref{lem:surj} it suffices to show that the divisor $N=D-A-B-C-K=-3K+G-(A+B+C)$ is nef and big. 

If $H$ is any exceptional curve and $A,B,C$ are curves in one of the above configurations then $H\cdot(A+B+C)\leq 3$ (since the graph of exceptional curves does not contain triangles with multiple edges), so $N\cdot H\geq 3-3=0$. 

To show that $N$ is big we compute
\[
N^2=(-3K+G-A-B-C)=(G+A'+B'+C')^2\geq (A'+B'+C')^2
\]
Now, the subgraph spanned by $A', B'$ and $C'$ is isomorphic to the subgraph spanned by $A, B$ and $C$ so $(A'+B'+C')^2\geq 1$.\end{proof}

\begin{lemma} \label{lem:F^2=0} The graph $G_7$ is $2$-capturable with the above moves starting with the subgraph $e_1,e_2$.
\end{lemma}
\begin{proof} We will describe the stages of the game explicitly,
\begin{itemize}
\item{Start with $e_1$ and $e_2$}
\item{Capture $h_1$ and $h_2$ from $e_1$ and $e_2$ using move $(P1)$}
\item{Capture $e_3,\dots e_7$ from $h_1$ and $h_2$ using $(P2)$}
\item{Capture all remaining exceptional curves from the $e_i$ using move $(P2)$ (this is possible since every other curve is adjacent to some pair of $e_i$'s)}
\end{itemize}
\end{proof}
\begin{lemma} \label{lem:X7mDg1} For any Del Pezzo surface $X_7$ and any divisor $D$ with $m_D\geq 2$
\[
b_{1,D}({\rm Cox}(X_7))=0
\]
\end{lemma}
\begin{proof} The result follows immediately from Theorem~\ref{thm:games} and the last two lemmas.\end{proof}

\subsection{Divisors $D$ with $m_D=1$}  
\begin{definition} A divisor class $Q$ on $X_r$ is called a conic bundle if it satisfies
\[
-K\cdot Q=2\text{ ~ and ~ } Q^2=0
\]\end{definition}
We will focus on divisors $D\neq -K$ with $m_D=1$. These can be written as $D=-K+F$ where $F\neq 0$ is nef and effective. We will study two cases depending on whether or not $F$ is a (positive) multiple of a conic bundle.
\begin{lemma} \label{lem:F^2=0} If $F\neq 0$ is a nef and effective divisor on $X_r$, $r\geq 3$ then the following are equivalent:
\begin{itemize}
\item{$F$ contracts a conic bundle (i.e. two adjacent curves in $G_r$)}
\item{$F^2=0$}
\item{$F=mQ$ for some conic bundle $Q$ and $m>0$}
\end{itemize}
\end{lemma}
\begin{proof} If $r=2$ then, reasoning as in Lemma~\ref{lem:nefandeffective} we see that $F$ can be written as $F=a_0(L-E_1)+a_1L +a_2(L-E_2)$ with $a_i\geq 0$ thus the equivalence follows immediately. 

\noindent For $r>2$ we will study each hypothesis separately,
\begin{itemize}
\item{Suppose $F$ contracts a conic bundle $Q$. Let $V_1,V_2$ be exceptional curves such that $V_1+V_2=Q$. Then $F\cdot V_1=Q\cdot V_1=0$ and both $F$ and $Q$ are divisors on the Del Pezzo $X_{r-1}$ obtained by contracting $V_1$ and moreover $Q$ is still a conic bundle contracted by $F$,  continuing inductively we can reduce to the case $r=2$.} 
\item{If $F^2=0$ then there exists a curve $E$ such that $F\cdot E=0$ (else $F=-K+R$ for some nef and effective divisor $R$ so $F^2>0$). Thus $F$ is a divisor on the Del Pezzo $X_{r-1}$ with $F^2=0$ and we can reduce to the case $r=2$ by induction.}
\item{If $F=mQ$ then $F^2=0$ and the statement follows from the last bullet.}
\end{itemize}
\end{proof}


\subsubsection{Divisors $D=-K+F$, with $F$ not a multiple of a conic bundle and $m_D=1$} If $D$ is one of these then,
\begin{itemize}
\item{Since $F$ is nef and not a multiple of a conic bundle either $-K\cdot F\geq 3$ or $F=-K$. The second case cannot occur since otherwise $m_D=2$ hence $-K\cdot F\geq 3$.}  
\item{$F$ does not contract any pair of adjacent curves (else $F$ would contract a conic bundle contradicting Lemma~\ref{lem:F^2=0}).}
\end{itemize} 
Throughout this section let $C$ be an exceptional curve with $F\cdot C=0$ (such a curve exists since $m_D=1$) and let $A$ be any curve disjoint from $C$ with
\[
F\cdot A=\min\{F\cdot E: E \text{ is an exceptional curve disjoint from $C$}\}
\]

To simplify the exposition we will first discuss the general strategy and then prove the validity of the required moves,

\begin{lemma} \label{lem:moves01} Let $Q,B$ be any two exceptional curves in $X_7$ and let $A$ and $C$ be defined as above. 
The following capture moves are valid for the divisor $D$,
\begin{figure}[h]
\begin{center}
\scalebox {0.68}{
\includegraphics{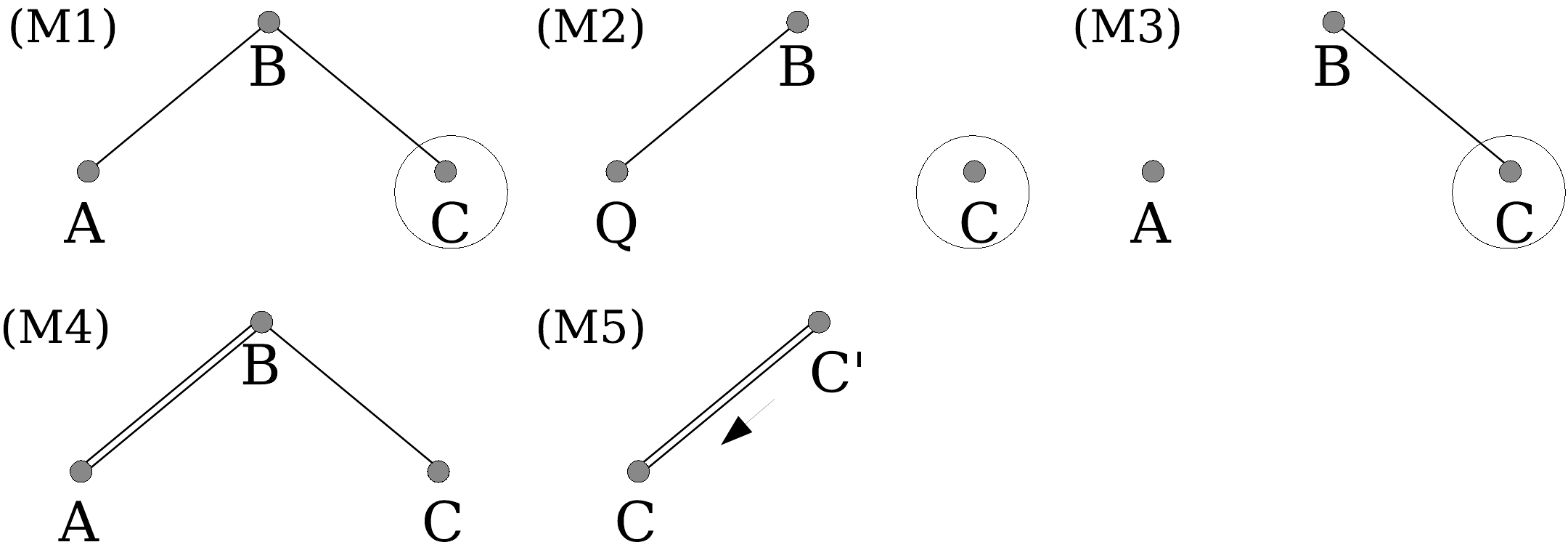}
}
\caption{Valid moves (M) for $r=7$ and $m_D=1$}
\end{center}
\end{figure}
\begin{itemize}
\item{(M1): Capturing every $B$ such that $A\cdot B=B\cdot C=1$ from $A$ and $C$}
\item{(M2): Capturing every $B$ such that there exists a $Q$ with $Q\cdot B=1$ and $Q\cdot C=B\cdot C=0$ from $C$ and $Q$}
\item{(M3): Capturing every $B$ such that $B\cdot C=1$ and $A\cdot B=0$ from $A$ and $C$}
\item{(M4): Capturing $A'$ from $A$ and $C$}
\item{(M5): Capture $C'$ with $C$}
\end{itemize}
\end{lemma}
\begin{proof} This follows from Lemma~\ref{lem:surj} by Lemmas~\ref{lem:configs1}, ~\ref{lem:M2} and~\ref{lem:M5} which we prove at the end of the section.\end{proof}
\begin{lemma} The graph of exceptional curves $G_7$ is $2$-capturable starting from the subgraph spanned by $A$, $C$ using the moves of Lemma~\ref{lem:moves01}.
\end{lemma}
\begin{proof} We will describe the stages of the game explicitly,
\begin{itemize}
\item{Start with $A$ and $C$}
\item{Capture all neighbors of $A$ and $C$ from $A$ and $C$ using moves $(M1),\dots ,(M5)$}
\item{Capture all remaining pieces using move $(M2)$. This is possible since every other exceptional curve is adjacent to some curve $H$ adjacent to $A$ (hence already captured) and disjoint from $C$ (necessary and sufficient conditions to use $(M2)$)}
\end{itemize}
\end{proof}
\begin{lemma} \label{lem:X7mDe11} For any Del Pezzo surface $X_7$ and any divisor $D=-K+F$ with $m_D=1$ such that $F\neq 0$ and $F$ not a multiple of a conic bundle,
\[
b_{1,D}({\rm Cox}(X_7))=0
\] 
\end{lemma}
\begin{proof} The result follows immediately from Theorem~\ref{thm:games} by the two Lemmas above.
\end{proof}

We will now show the validity of the moves $(M_1),\dots, (M_5)$. Recall that $F\cdot C=0$ and that $A$ was chosen so that it is disjoint from $C$ and 
\[
F\cdot A=\min\{F\cdot E: E \text{ is an exceptional curve disjoint from $C$}\}
\]

\begin{lemma} The following statements hold,
\label{lem:A'} 
\begin{itemize}
\item{If $F\cdot A=0$ then $F\cdot A'\geq 3$}
\item{If $F\cdot A\neq 0$ then $F\cdot A'\geq 2$}
\end{itemize}
\end{lemma}
\begin{proof} 
By definition of $A'$, $A+A'=-K$ so $F\cdot(A+A')=-K\cdot F\geq 3$ and the first statement follows.

For the second statement let $X_6$ be the cubic surface obtained by blowing down $A$. In $G_6$ every vertex belongs to a triangle so in $G_7$ there is a triangle with vertices $Q_1,Q_2, C$ (disjoint from $A$) with $Q_1+Q_2+C=-K_6=-K_7+A$. Thus
\[
F\cdot(Q_1+Q_2+C)=F(-K+C)=-K\cdot F\geq 3
\]  
Hence there exists a curve  $Q_1$ adjacent to $C$ with $F\cdot Q_1\geq 2$. 
It's dual curve $Q_1'$ is disjoint from $C$ so $F\cdot A\leq F\cdot Q_1'$ and  $F(A+A')=F(Q_1+Q_1')=-K\cdot F$ so $F\cdot A'\geq F\cdot Q_1=2$ as we wanted to show.
\end{proof}

\begin{lemma} \label{lem:configs1} If $B$ is an exceptional curve in any of the following configurations, then $h^1(D-A-B-C)=0$.
\begin{itemize}
\item{(M1): $A\cdot B=B\cdot C=1$}
\item{(M3): $B\cdot C=1$ and $A\cdot B=0$}
\item{(M4): $B=A'$} 
\end{itemize}
\end{lemma}
\begin{proof} We will show that the divisor $N=D-A-B-C-K=-2K+F-(A+B+C)$ is nef and big. The result will follow from the Kawamata-Viehweg vanishing theorem.
\begin{itemize}
\item{If $H\cdot(A+B+C)\leq 2$ then $N\cdot H\geq 0$}
\item{If $H\cdot(A+B+C)=3$ then $H$ is either adjacent to $C$ (so $F\cdot H\geq 1$ since $F$ does not contract any conic bundle) or the configuration is $(M3)$ and $H=B'$ (In this case $F\cdot H\geq 1$ since either $F\cdot A=0$ and $F$ contracts no conic bundle or $F\cdot A\geq 1$ and every other curve $Q$ disjoint from $C$, and in particular $B$, has $F\cdot Q\geq 1$). Hence $N\cdot H\geq 2+1-3=0$}
\item{If $H\cdot(A+B+C)=4$ then $H=A'$ and by Lemma~\ref{lem:A'} $F\cdot A'\geq 2$ so $N\cdot H\geq 2+2-4=0$}  
\end{itemize}
Thus $N$ is a nef divisor. To verify bigness we will compute $N^2$
\[
N^2=4K^2+2(-2K(F-(A+B+C)))+(F-(A+B+C))^2=
\]
\[
=4K^2+2(-2KF-6)+F^2-2FA-2FB+(A+B+C)^2=
\]
\[
= -4+2((-K-A)F+(-K-B)F)+F^2+(A+B+C)^2 \geq 
\]
\[
\geq -3+2((-K-A)F+(-K-B)F)+(A+B+C)^2=-3+2(A'F+B'F)+(A+B+C)^2
\]
where the inequality follows from the fact that $F$ is not a multiple of a conic bundle. Now we will study the last quantity in each of the configurations.
\begin{itemize}
\item{In $(M1)$ or $(M3)$, $F\cdot(A'+B')\geq 3$ since either $F\cdot A=0$ and $F\cdot A'\geq 3$ or $F\cdot A\geq 1$ and $F\cdot B'\geq 1$ (since $B'$ is disjoint from $C$) and $F\cdot A'\geq 2$ by Lemma~\ref{lem:A'}. Moreover $(A+B+C)^2\geq -1$ so $N^2\geq 2>0$}
\item{In $(M4)$, $A'+B'=-K$ and $(A+B+C)^2=3$ so $N^2\geq 6>0$} 
\end{itemize}
\end{proof}

\begin{lemma} \label{lem:M2} If $Q$ and $B$ are exceptional curves such that $Q\cdot B=1$ and $Q\cdot C=B\cdot C=0$ (see (M2) in the figure) then $h^1(D-Q-B-A)=0$.
\end{lemma}
\begin{proof} As before we will show that $N=D-Q-B-C-K=-2K+F-(Q+B+C)$ is nef and effective.\\
For an exceptional curve $H$ we will study several cases according to the value of $H\cdot (A+B+C)$,
\begin{itemize}
\item{ If $H\cdot(Q+B+C)\leq 2$ then $N\cdot H\geq 2-2=0$}
\item{ If $H\cdot(Q+B+C)=3$ then $H$ must intersect $C$ so $F\cdot H\geq 1$ (since $F$ contracts no conic bundle) and $N\cdot H\geq 0$}
\item{Finally, if $H\cdot(Q+B+C)=4$ then $H=C'$ so $F\cdot H=F\cdot C'=F\cdot(C'+C)=-K\cdot F\geq 3$ (since $F$ is not a conic bundle nor $-K$) so $N\cdot H\geq 0$}
\end{itemize} 
To see that $N$ is big note that (as in Lemma~\ref{lem:configs1})
\[
N^2\geq -3+2(Q'F+B'F)+(Q+B+C)^2
\]
In our case $Q', B'$ and $C$ form a triangle so $F(Q'+B')=F(Q'+B'+C)=F(-K+V)\geq -KF\geq 3$ (where $V$ is disjoint from $Q', B'$ and $C$). Since $(Q+B+C)^2=-1$, $N^2\geq 2>0$ and by Kawamata-Viehweg, $h^1(N+K)=0$.\end{proof}

\begin{lemma} \label{lem:M5} If $D\cdot C=1$ then every section of $H^0(X_r,D-C')$ is divisible by $C$, that is, $C$ is in the fixed part of $D-C'$.
\end{lemma}
\begin{proof} In this case $(D-C')\cdot C=-1$ so the result follows form the long exact sequence in cohomology associated to the short exact sequence of sheaves
\[
0\rightarrow \Osh_{X_r}[D-C-C']\rightarrow \Osh_{X_r}[D-C']\rightarrow \Osh_C[D-C']\rightarrow 0
\]
\end{proof}

\subsubsection{ Divisors $D=-K+F$ with $m_D=1$ and $F=mQ$ a positive multiple of a conic bundle $Q$} Throughout the rest of the section we fix disjoint exceptional curves $A$ and $C$ contracted by $F$ (such curves exist since $F$ is a multiple of a conic bundle). 

We will first overview the strategy and then prove the validity of the required moves.
\begin{lemma}\label{lem:moves02} Let $B$ be any exceptional curve in $X_7$ and let $A$ and $C$ be defined as above. 
The following capture moves are valid for the divisor $D$,
\begin{figure}[h]
\begin{center}
\scalebox {0.75}{
\includegraphics{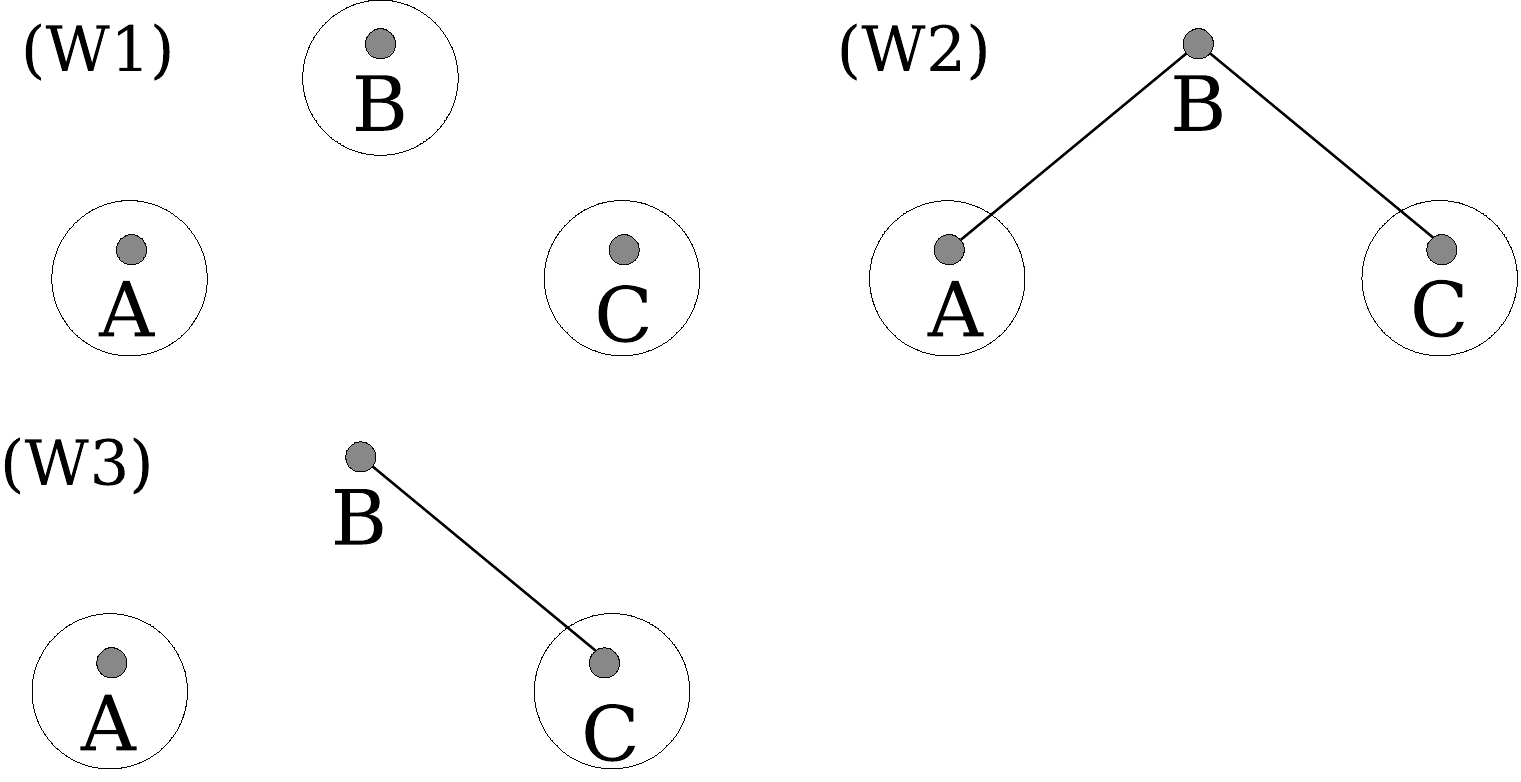}
}
\caption{Valid moves (W) for $r=7$ and $m_D=1$}
\end{center}
\end{figure}
\begin{itemize}
\item{(W1): Capturing every $B$ contracted by $F$ such that $A\cdot B=B\cdot C=0$ from $A$ and $C$}
\item{(W2): Capturing every $B$ such that $A\cdot B=B\cdot C=1$ from $A$ and $C$}
\item{(W3): Capturing every $B$ such that $F\cdot B'\geq 1$, $B\cdot C=1$ and $A\cdot B=0$ from $A$ and $C$}
\end{itemize}
\end{lemma}
\begin{proof} The validity of $(W1),\dots, (W3)$ follows from Lemma~\ref{lem:surj} and Lemma~\ref{lem:configs2}.\end{proof}

\begin{lemma} The graph of exceptional curves $G_7$ is $2$-capturable starting from the subgraph spanned by $A$, $C$ using the moves of Lemma~\ref{lem:moves02}.\end{lemma}
\begin{proof} We will describe the stages of the game explicitly,
\begin{itemize}
\item{Start with any two disjoint exceptional curves $A$ and $C$ contracted by $F$}
\item{Capture all the curves $B$ contracted by $F$ from $A$ and $C$ as follows:
\begin{itemize}
\item{If $F$ is disjoint from $A$ and $C$ use $(W1)$}
\item{If $F$ is adjacent to $A$ and $C$ use $(W2)$}
\item{If $F$ is adjacent to only one of $A$ and $C$ then use $(W3)$ (note that since $B$ is contracted by $F$ then $F\cdot B'=F\cdot(B+B')=-K\cdot F\geq 2$ so the required condition for move $(W3)$ is satisfied).} 
\end{itemize}
} 
\item{Capture all remaining exceptional curves using move $(W2)$. This is possible since by Lemma~\ref{lem:Qcontracted} below every other exceptional curve is adjacent to a pair of disjoint curves contracted by $F$.}
\end{itemize}
\end{proof}
\begin{lemma} \label{lem:X7mDe12} For any Del Pezzo surface $X_7$ and any divisor $D=-K+F$ with $F=mQ$ a (positive) multiple of a conic bundle we have,
\[
b_{1,D}({\rm Cox}(X_7))=0
\] 
\end{lemma}
\begin{proof} this follows immediately from Theorem~\ref{thm:games} by the two Lemmas above.
\end{proof}
Now we will prove the validity of the required capture moves
\begin{lemma} \label{lem:configs2} If $A$ and $C$ are disjoint exceptional curves contracted by $F$ and $B$ is an exceptional curve in any of the following configurations, then $h^1(D-A-B-C)=0$.
\begin{itemize}
\item{(W1): $A\cdot B=B\cdot C=0$}
\item{(W2): $A\cdot B=B\cdot C=1$}
\item{(W3): $F\cdot B'\geq 1$, $B\cdot C=1$ and $A\cdot B=0$} 
\end{itemize}
\end{lemma}
\begin{proof} We will show that the divisor $N=D-A-B-C-K=-2K+F-(A+B+C)$ is nef and big. The result will then follow from the Kawamata-Viehweg vanishing theorem.
\begin{itemize}
\item{If $H\cdot(A+B+C)\leq 2$ then $N\cdot H\geq 0$}
\item{If $H\cdot(A+B+C)=3$ then  either $H$ intersects two contracted curves (so $F\cdot H\geq 1$ since $Q$ contracts only one conic bundle) or $H$ is dual to a contracted curve (and $F\cdot H\geq 2$) or we are in move $(W3)$ and $H=B'$ (so by assumption $F\cdot H\geq 1$). In all cases $N\cdot H\geq 2+1-3= 0$}
\item{If $H\cdot(A+B+C)=4$ then $H$ is dual to some contracted curve so $F\cdot H=F\cdot (H+H')=-K\cdot F\geq 2$ so $N\cdot H\geq 2+2-4=0$}
\end{itemize}
Thus $N$ is a nef divisor. To verify bigness we will compute $N^2$
\[
N^2=(F+(-K-A)+(-K-B)-C)^2=(F+A'+B'-C)^2=
\]
\[
=2F(A'+B')+(A'+B'-C)^2
\]
Now we will study the last quantity in each of the configurations.
\begin{itemize}
\item{In $(W1)$, $2F(A'+B')=-4KF\geq 8$ and $(A'+B'-C)^2=-7$}
\item{In $(W2)$, $2F(A'+B')=-2KF+2FB'\geq 4$ and $(A'+B'-C)^2=-3$}
\item{In $(W3)$, $2F(A'+B')=-2KF+2FB'\geq 6$ (since by assumption $F\cdot B'\geq 1$) and $(A'+B'-C)^2=-5$}
\end{itemize}
In all cases $N^2\geq 1>0$.
\end{proof}
\begin{lemma} \label{lem:Qcontracted} For every exceptional curve $B$ not contracted by $F$ there exists a pair of disjoint exceptional curves $A$ and $C$ contracted by $F$ adjacent to $B$.\end{lemma}
\begin{proof} Write $F=mQ$ for some conic bundle $Q$. Let $\{W_1,W_2\}$ and $\{V_1,V_2\}$ be distinct sets of curves such that $Q=W_1+W_2=V_1+V_2$ and note that the $W's$ and the $V's$ are orthogonal (since $Q^2=0$). Now $Q\cdot B>0$ implies that at least one $W_i$ (say $W_1$) and one $V_i$ (say $V_1$) intersect $B$ (singly). Since $Q\cdot W_1\leq Q\cdot(W_1+W_2)=0$ and similarly $Q\cdot V_1=0$ the curves $W_1$ and $V_1$ are contracted by $Q$ and the result follows.\end{proof}

\section{A proof of the conjecture of Batyrev and Popov}
In this final section we put all our results together and prove the conjecture of Batyrev and Popov,
\begin{theorem} For $4\leq r\leq 7$ and any choice of points $\{p_1,\dots, p_r\}$ in $\mathbb{P}^2$ such that $X_r(p_1,\dots, p_r)$ is a Del Pezzo surface, 
\[
{\rm Cox}(X_r)=k[V_r]/Q_r(p_1,\dots, p_r)
\]
In other words the Cox rings of Del Pezzo surfaces are quadratic algebras.\end{theorem}
\begin{proof} Use induction on $r$. The base case $r=4$ was settled in~\cite{BP}. Assume that $r>4$ and note that
\begin{enumerate}
\item{By Lemma~\ref{lem:mingens} the minimal generators of the ideal defining the Cox ring have degrees $D$ which are nef and effective.}
\item{By Lemma~\ref{lem:DcontractsCurves}, $b_{1,D}({\rm Cox}(X_r))=0$ for every nef and effective divisor of degree $\geq 2$ which contracts exceptional curves (the induction hypothesis is used in the proof of this Lemma).}
\item{If $D$ is any nef divisor of degree $-K\cdot D\geq 3$ which does not contract curves then:
\begin{enumerate}
\item{If $r=5,6$ then Lemma~\ref{lem:X6mD0} shows that $b_{1,D}({\rm Cox}(X_r))=0$.}
\item{If $r=7$ then Lemmas~\ref{lem:X7mDg1}, ~\ref{lem:X7mDe11} and~\ref{lem:X7mDe12} exhaust all possibilities and show that $b_{1,D}({\rm Cox}(X_r))=0$.}
\end{enumerate}
} 
\end{enumerate}
Hence the ideal $I_r(p_1,\dots, p_r)$ has all its minimal generators in anticanonical degree $2$ and the Theorem follows.
\end{proof}

\end{document}